\documentclass[10pt, twocolumn,twoside]{IEEEtran}
\usepackage[T1]{fontenc}

\IEEEoverridecommandlockouts                          

\usepackage{amsmath} 
\usepackage{amssymb}  
\usepackage{amsthm}
\usepackage{xcolor}
\usepackage{caption}
\usepackage{graphicx}
\usepackage{mathtools}
\usepackage{caption}
\usepackage{subcaption}
\usepackage{floatrow}
\usepackage{tikz}
\usepackage{url}

\usepackage{algorithmicx}
    \usepackage[ruled]{algorithm}
    \usepackage{algpseudocode}

\usepackage{pgfplots}
\usepackage{xcolor}
\usetikzlibrary{matrix,arrows,calc,positioning,shapes,decorations.pathreplacing}
\usepackage{graphicx}

\usepackage{amsmath} 

\usepackage{enumerate}
\usepackage[all,tips]{xy}
\SelectTips{cm}{11}

\newtheorem{theorem}{Theorem}

\newtheorem{definition}{Definition}
\newtheorem{remark}{Remark}
\newtheorem{proposition}{Proposition}
\newtheorem{lemma}{Lemma}
\newtheorem{example}{Example}

\newcommand{\R}{\mathbb{R}}

\newcommand{\dotx}{\dot{x}}
\newcommand{\doty}{\dot{y}}
\newcommand{\dotz}{\dot{z}}

\DeclareMathOperator{\spn}{span}

\begin{document}
\title{On the Geometry of Virtual Nonlinear Nonholonomic Constraints}
\author{Efstratios Stratoglou, Alexandre Anahory Simoes, Anthony Bloch \IEEEmembership{Fellow, IEEE}, Leonardo J. Colombo. \IEEEmembership{Member, IEEE}
\thanks{E. Stratoglou is with Universidad Polit\'ecnica de Madrid (UPM), José Gutiérrez Abascal, 2, 28006 Madrid, Spain. (e-mail: ef.stratoglou@alumnos.upm.es).}
\thanks{A. Anahory Simoes is with the School of Science and Technology, IE University, Spain. (e-mail: alexandre.anahory@ie.edu).}
\thanks{A. Bloch is with Department of Mathematics, University of Michigan, Ann Arbor, MI 48109, USA. (e-mail: abloch@umich.edu)}
\thanks{L. Colombo is with Centre for Automation and Robotics (CSIC-UPM), Ctra. M300 Campo Real, Km 0,200, Arganda del Rey - 28500 Madrid, Spain. (e-mail: leonardo.colombo@csic.es)}
\thanks{The authors acknowledge financial support from Grant PID2022-137909NB-C21 funded by MCIN/AEI/ 10.13039/501100011033 and the LINC Global project from CSIC "Wildlife Monitoring Bots" INCGL20022. A.B. was partially supported by NSF grant  DMS-2103026, and AFOSR grants FA
9550-22-1-0215 and FA 9550-23-1-0400}}

\maketitle

\begin{abstract}
 Virtual constraints are relations imposed on a control system that become invariant via feedback control, as opposed to physical constraints acting on the system.  Nonholonomic systems are mechanical systems with non-integrable constraints on the velocities. In this work, we introduce the notion of \textit{virtual nonlinear nonholonomic constraints} in a geometric framework which is a controlled invariant submanifold and we show the existence and uniqueness of a control law preserving this submanifold. We illustrate the theory with various examples and present simulation results for an application.
\end{abstract}


\section{Introduction}
\IEEEPARstart{V}irtual constraints are relations on the configuration variables of a control system which are imposed through feedback control and the action of actuators, instead of through physical connections such as gears or contact conditions with the environment. The advantage of working with virtual constraints is that they can be re-programmed instantaneously without any change to the connections of the links of a robot or its environment. As a consequence, one may achieve a desired prescribed motion by imposing virtual constraints.  Virtual constraints extend the application of zero dynamics to feedback design (see e.g., \cite{Isidori}, \cite{Westervelt et al 2018}). 

Virtual holonomic constraints have been studied over the past few years in a variety of contexts, such as motion planning and control \cite{Freidovich et al}, \cite{Shiriaev et al}, \cite{Mohammadi et al}, \cite{Westerberg et al} and biped
locomotion where it was used to achieve a desired walking gait \cite{Chevallereau et al}, \cite{Westervelt et al}. Virtual nonholonomic constraints are a class of virtual constraints that depend on velocities rather than only on the configurations of the system. Those virtual constraints were introduced in \cite{griffin2015nonholonomic} to design a velocity-based swing foot placement in bipedal robots. In particular, this class of virtual constraints was used in \cite{hamed2019nonholonomic,horn2020nonholonomic,horn2018hybrid,horn2021nonholonomic} to encode velocity-dependent stable walking gaits via momenta conjugate to the unactuated degrees of freedom of legged robots and prosthetic legs.   

The recent work \cite{moran2021energy} (see also \cite{moran2023}) introduces an approach to defining rigorously virtual nonholonomic constraints, but it is not set in the most appropriate geometric setting to study this kind of constraint: that of tangent bundles. In their study the authors make no distinction between making a constraint invariant under the closed-loop system or being stabilized by it. In our work, we only consider a constraint to be a virtual constraint when it is invariant under the controlled motion.

In the paper \cite{virtual}, we developed a geometric description of linear virtual nonholonomic constraints, i.e., constraints that are linear in the velocities, while in \cite{affine} we addressed the problem of affine virtual nonholomonic constraints, but the nonlinear case was not addressed because the nature of the constraints makes a thorough mathematical analysis difficult. In the present work, we extend the latest outcomes by laying the geometric foundations of virtual nonlinear nonholonomic constraints and studying their properties. We ensure the existence and uniqueness of a control law that makes the constraints invariant and we explore some consequences for the corresponding close-loop system. In addition, we check under which conditions the closed-loop dynamics coincides with the nonholonomic dynamics under the nonlinear constraints. Lastly, we give an explicit application of the theory to the motion of
particles moving with an alignment on the velocities and we test our results with numerical simulations.

The remainder of the paper is structured as follows. In Section II we present the necessary background for mechanical systems on Riemannian manifolds and recall the equations of motion for a nonlinear nonholonomic mechanical system. In Section III we give a geometric construction of virtual nonholonomic constraints. The main result of the paper is Theorem \ref{main:theorem}, where under some assumptions, we prove the existence and uniqueness of a control law making the constraints invariant under the closed-loop system. Additionally, we examine the geometric properties of the closed-loop dynamics and exemplify it in examples. Next, in Section IV we present an application of our results to the motion of
particles moving with an alignment on the velocities, where we enforce virtual nonlinear nonholonomic constraint satisfying the assumptions of Theorem \ref{main:theorem} and simulate its behavior.


\section{Mechanical Systems on Riemannian Manifolds}

In this section, we will review the equations of motion for mechanical systems subject to nonlinear constraints.

Suppose $Q$ is a differentiable manifold of dimension $n$. Throughout the text, $q^{i}$ will denote a particular choice of local coordinates on this manifold and $TQ$ denotes its tangent bundle, with $T_{q}Q$ denoting the tangent space at a specific point $q\in Q$ generated by the coordinate vectors $\frac{\partial}{\partial q^{i}}$. Usually $v_{q}$ denotes a vector at $T_{q}Q$ and, in addition, the coordinate chart $q^{i}$ induces a natural coordinate chart on $TQ$ denoted by $(q^{i},\dot{q}^{i})$. There is a canonical projection $\tau_{Q}:TQ \rightarrow Q$, sending each vector $v_{q}$ to the corresponding base point $q$. Note that in coordinates $\tau_{Q}(q^{i},\dot{q}^{i})=q^{i}$. The tangent map of the canonical projection is given by $T\tau_Q:TTQ\to TQ.$ The cotangent bundle of $Q$ is denoted by $T^*Q$ and for $q\in Q$ the cotangent space $T^*_qQ$ is generated by cotangent vectors $dq^i$ which satisfies the dual pairing $\langle dq^i,\frac{\partial}{\partial q^j}\rangle=\delta_{ij}$, where $\delta_{ij}$ is the Kronecker delta.

A vector field $X$ on $Q$ is a map assigning to each point $q\in Q$ a vector tangent to $q$, that is, $X(q)\in T_{q}Q$. In the context of mechanical systems, we find a special type of vector fields that are always defined on the tangent bundle $TQ$, considered as a manifold itself. A second-order vector field (SODE) $\Gamma$ on the tangent bundle $TQ$ is a vector field on the tangent bundle satisfying the property that $T\tau_{Q}\left(\Gamma (v_{q})\right) = v_{q}$. The expression of any SODE in coordinates is the following:
$$\Gamma(q^{i},\dot{q}^{i})= \dot{q}^{i}\frac{\partial}{\partial q^{i}} + f^{i}(q^{i},\dot{q}^{i}) \frac{\partial}{\partial \dot{q}^{i}},$$
where $f^{i}:TQ \rightarrow \mathbb{R}$ are $n$ smooth functions. We denote the set of all vector fields on $Q$ by $\mathfrak{X}(Q)$.

A one-form $\alpha$ on $Q$ is a map assigning to each point $q$ a cotangent vector to $q$, that is, $\alpha(q)\in T^{*}Q$. Cotangent vectors acts linearly on vector fields according to $\alpha(X) = \alpha_{i}X^{i}\in \mathbb{R}$ if $\alpha = \alpha_{i}dq^{i}$ and $X = X^{i} \frac{\partial}{\partial q^{i}}$. In the following, we will refer to two-forms or $(0,2)$-tensor fields which are bilinear maps that act on a pair of vector fields to produce a number and also to $(1,1)$-tensor fields which are linear maps that act on a vector field to produce a new vector field.

A symplectic form $\omega$ on a manifold $M$ is a $(0,2)$-type tensor field that is skew-symmetric and non-denerate, i.e., $\omega(X,Y)=-\omega(Y,X)$ for all vector fields $X$ and $Y$ and if $\omega(X,Y)=0$ for all vector fields $X$ then $Y=0$.

The symplectic form induces a linear isomorphism $\flat_{\omega}:\mathfrak{X}(M)\rightarrow \Omega^{1}(M)$, given by $\langle\flat_{\omega}(X),Y\rangle=\omega(X,Y)$ for any vector fields $X, Y$. The inverse of $\flat_{\omega}$ will be denoted by $\sharp_{\omega}$.

In the following, we will use the canonical almost tangent structure $J:TTQ \rightarrow TTQ$. This is a type $(1,1)$- tensor field on $TQ$ whose expression in local coordinates is
$J=dq^{i}\otimes \frac{\partial}{\partial \dot{q}^{i}}$, where $\otimes$ stands for the tensor product. For instance, if $\Gamma$ is a SODE vector field $J(\Gamma) = \dot{q}^{i}\frac{\partial}{\partial \dot{q}^{i}}$.

Given a Lagrangian function $L:TQ\rightarrow \mathbb{R}$, the associated energy $E_{L}$ is the function defined by $E_{L}(q,\dot{q})=\dot{q}\frac{\partial L}{\partial \dot{q}} - L(q,\dot{q})$ and we may write a symplectic form on $TQ$, denoted by $\omega_{L}$, defined by $\omega_L = -d(J^{*}dL)$. In natural coordinates of $TQ$, $\omega_{L}=\frac{\partial^{2} L}{\partial \dot{q}^{i} \partial q^{j}} dq^{i}\wedge dq^{j} + \frac{\partial^{2} L}{\partial \dot{q}^{i} \partial \dot{q}^{j}} dq^{i}\wedge d\dot{q}^{j}$. This geometric construction is used to write Euler-Lagrange dynamics as the integral curves of the vector field $\Gamma_{L}$ solving the equation $i_{\Gamma_{L}}\omega_{L}=dE_{L}$, where $i_{\Gamma_{L}}\omega_{L}$ denotes the contraction of $\Gamma_L$ and $\omega_L$ (see \cite{B}). In fact, this is the geometric equation defining Hamiltonian vector fields in general symplectic manifolds.

Before proceeding, we will recall the definition of Riemannian metric. A Riemannian metric is a generalization of the inner product on a vector space to arbitrary manifolds. In fact, one can describe it as an inner product in each tangent space $T_{q}Q$ that varies smoothly with the base point $q$. In particular, since the metric will be an inner product on each tangent space, as will see below, it will be defined on the space $TQ\times_{Q}TQ$, composed of pairs of tangent vectors lying in the same tangent space. In this way, we avoid defining the inner product between two vectors that are tangent at different points. More precisely,
\begin{definition}
    A Riemannian metric $\mathcal{G}$ on a manifold $Q$ is a $(0,2)$-tensor, i.e., a bilinear map $\mathcal{G}:TQ\times_{Q} TQ \rightarrow \mathbb{R}$, satisfying the following properties:
    \begin{enumerate}
        \item[(i)] symmetric:  $\mathcal{G}(v_{q},w_{q})=\mathcal{G}(w_{q},v_{q})$ for all $q\in Q$ and $v_{q}$,$w_{q} \in T_{q}Q$.
        \item[(ii)] non-degenerate: $\mathcal{G}(v_{q},w_{q})=0$ for all $w_{q}\in TQ$ if and only if $v_{q}=0$. 
        \item[(iii)] positive-definite: $\mathcal{G}(v_{q},v_{q})\geqslant 0$, with equality holding only if $v_{q}=0$.
    \end{enumerate}
    Accordingly, if $\mathcal{G}$ is a Riemannian metric then the pair $(Q,\mathcal{G})$ is called a Riemannian manifold.
\end{definition}

If $(q^{1},\dots,q^{n})$ are local coordinates on $Q$, then the local expression of the Riemannian metric $\mathcal{G}$ is
$\displaystyle{
	\mathcal{G}=\mathcal{G}_{i j} dq^{i} \otimes dq^{j}}$ with  $\displaystyle{\mathcal{G}_{i j}=\mathcal{G}\left(\frac{\partial}{\partial q^{i}},\frac{\partial}{\partial q^{j}}\right)}$.

In the following, we will make use of a special technique to lift a Riemannian metric on a manifold to a metric on the tangent bundle $TQ$. The complete lift of a Riemannian metric $\mathcal{G}$ on $Q$ is denoted by $\mathcal{G}^{c}$ and it is almost a Riemannian metric, since it does not satisfy property (iii) from the definition above, i.e., it is not positive-definite, which is similar with what happens in special relativity (see for instance \cite{schutz}) where the metric 
is indefinite. Given natural bundle coordinates on $TQ$, its local expression is $\displaystyle{\mathcal{G}^{c}=\dot{q}^{k}\frac{\partial \mathcal{G}_{ij}}{\partial q^{k}} dq^{i} \otimes dq^{j} + \mathcal{G}_{i j} dq^{i} \otimes d\dot{q}^{j} + \mathcal{G}_{i j} d\dot{q}^{i} \otimes dq^{j}}$.

For the Riemannian metric $\mathcal{G}$ on $Q$, we can use its non-degeneracy property to define the musical isomoprhism $\flat_{\mathcal{G}}:\mathfrak{X}(Q)\rightarrow \Omega^{1}(Q)$ defined by $\flat_{\mathcal{G}}(X)(Y)=\mathcal{G}(X,Y)$ for any $X, Y \in \mathfrak{X}(Q)$. Also, denote by $\sharp_{\mathcal{G}}:\Omega^{1}(Q)\rightarrow \mathfrak{X}(Q)$ the inverse musical isomorphism, i.e., $\sharp_{\mathcal{G}}=\flat_{\mathcal{G}}^{-1}$.



\begin{definition}
    The vertical lift of a vector field $X\in \mathfrak{X}(Q)$ to $TQ$ is defined by $$X_{v_{q}}^{V}=\left. \frac{d}{dt}\right|_{t=0} (v_{q} + t X(q)).$$
    The  complete lift of a vector field, $X$, which in local coordinates is given by $X=X^i\frac{\partial}{\partial q^i}$ is $$X^c=X^i\frac{\partial}{\partial q^i} + \dot{q}^j\frac{\partial X^i}{\partial q^j}\frac{\partial}{\partial \dot{q}^i}.$$ The vertical lift of a one-form $\alpha\in\Omega^1(Q)$ is defined as the pullback of $\alpha$ to $TQ$, i.e. $$\alpha^V=(\tau_Q)^*\alpha,$$ which locally is $\alpha^V=\alpha_idq^i$ and its complete lift is $$\alpha^c=\dot{q}^j\frac{\partial\alpha^i}{\partial q^j}dq^i + \alpha^idq^i.$$
\end{definition}
\begin{proposition}
    For a Riemannian metric $\mathcal{G}$ on $Q$, vector  fields  $X,Y\in\mathfrak{X}(Q)$ and a one-form $\alpha\in\Omega^1(Q)$ we have
    $$(\alpha(X))^V=\alpha^c(X^V),$$
    $$\mathcal{G}^c(X^V,Y^c) = \mathcal{G}^c(X^c,Y^V) = [\mathcal{G}(X,Y)]^V,$$
    $$\mathcal{G}^c(X^V,Y^V)=0.$$
\end{proposition}
For details of the aforementioned one can see \cite{Leon_Rodrigues}.

The complete lift of a Riemannian metric possess useful properties such as the one described in the following lemma.

\begin{lemma}\label{completemetricLemma}
    Let $(Q,\mathcal{G})$ be a Riemannian manifold and $\alpha\in \Omega^{1}(Q)$ a one-form. Then,
    $$\left[ \sharp_{\mathcal{G}}(\alpha)\right]^{V} = \sharp_{\mathcal{G}^{c}}(\alpha^{V}).$$
\end{lemma}

\begin{proof}
    Given any $Y\in\mathfrak{X}(Q)$, it is enough to prove the equality using the inner product with the lifts $Y^{c}$ and $Y^{V}$, because if $\{Y^{a}\}$ was a local basis of vector fields, then $\{(Y^{a})^{c}, (Y^{a})^{V}\}$ would also be a local basis of vector fields on $TQ$.

    On one hand, $$\mathcal{G}^{c}\left(\left[ \sharp_{\mathcal{G}}(\alpha)\right]^{V},Y^{V} \right) = 0 = \alpha^{V}(Y^{V}) = \mathcal{G}^{c}\left(\sharp_{\mathcal{G}^{c}}(\alpha^{V}),Y^{V} \right).$$

    On the other hand, $$\mathcal{G}^{c} \left(\left[ \sharp_{\mathcal{G}}(\alpha)\right]^{V}, Y^{c} \right) = \left[ \mathcal{G}(\sharp_{\mathcal{G}}(\alpha), Y)\right]^{V} = \left[ \alpha(Y)\right]^{V}$$
    
    $$= \alpha^{V}(Y^{c})=\mathcal{G}^{c} \left( \sharp_{\mathcal{G}^{c}}(\alpha^{V}),Y^{c} \right).$$
    Hence, the results follows by non-degeneracy of $\mathcal{G}^{c}$. 
\end{proof}

Finally, we recall the concept of a linear connection on a manifold that is used to generalize the concept of directional derivative of a vector field along another to a manifold. Formally, a linear connection on a manifold $Q$ is any map of the form $\nabla:\mathfrak{X}(Q)\times \mathfrak{X}(Q) \rightarrow \mathfrak{X}(Q)$ which is $C^{\infty}(Q)$-linear on the first factor,  $\mathbb{R}$-linear in the second factor, and if we denote the image of $X, Y \in \mathfrak{X}(Q)$ by $\nabla_{X} Y$, then $\nabla$ satisfies the Leibniz differentiation rule, i.e., 
$\nabla_{X} (f Y)=X(f)\cdot Y+f\cdot \nabla_{X} Y$ for every $f\in C^{\infty}(Q)$. In local coordinates, connections are fully described by the Chrystoffel symbols which are real-valued functions on $Q$ given by
\begin{equation*}
	\nabla_{\frac{\partial}{\partial q^{i}}}\frac{\partial}{\partial q^{j}}=\Gamma_{i j}^{k}\frac{\partial}{\partial q^{k}}.
\end{equation*}
Thus if $X$ and $Y$ are vector fields whose coordinate expressions are $X=X^{i}\frac{\partial}{\partial q^{i}}$ and $Y=Y^{i}\frac{\partial}{\partial q^{i}}$, then
\begin{equation*}
	\nabla_{X} Y=\left(X^{i}\frac{\partial Y^{k}}{\partial q^{i}}+X^{i}Y^{j} \Gamma_{i j}^{k}\right)\frac{\partial}{\partial q^{k}}.
\end{equation*}

In a Riemannian manifold, there is a special linear connection--the \textit{Levi-Civita connection}--associated to the Riemannian metric $\mathcal{G}$. This is the unique connection $\nabla^{\mathcal{G}}:\mathfrak{X}(Q)\times \mathfrak{X}(Q) \rightarrow \mathfrak{X}(Q)$  satisfying the following two additional properties:
\begin{enumerate}
\item[(i)] $[ X,Y]=\nabla_{X}^{\mathcal{G}}Y-\nabla_{Y}^{\mathcal{G}}X$ (symmetry)
\item[(ii)] $X(\mathcal{G}(Y,Z))=\mathcal{G}(\nabla_{X}^{\mathcal{G}}Y,Z)+\mathcal{G}(Y,\nabla_{X}^{\mathcal{G}}Z)$ (compatibillity of the metric).
\end{enumerate}

We might also introduce the covariant derivative of a vector field along a curve. The covariant derivative of a vector field $X\in \mathfrak{X}(Q)$ along a curve $q:I\rightarrow Q$, where $I$ is an interval of $\mathbb{R}$, is given by the local expression
\begin{equation*}
		\nabla_{\dot{q}}X (t)=\left( \dot{X}^{k}(t)+\dot{q}^{i}(t) X^{j}(t)\Gamma_{i j}^{k}(q(t)) \right)\frac{\partial}{\partial q^{k}}.
\end{equation*}
A geodesic in a Riemannian manifold is the curve of minimum length joining two points in space. Geodesics are characterized by the equation $\nabla_{\dot{q}} \dot{q}=0$.

\subsection{Nonlinear nonholonomic mechanics}

A nonlinear nonholonomic constraint on a mechanical system is a submanifold $\mathcal{M}$ of the tangent bundle $TQ$ from which the velocity of the system can not leave. Mathematically, the constraint may be written as the set of points where a function of the type $\phi:TQ \rightarrow \mathbb{R}^{m}$ vanishes, where $m < n=\dim Q$. That is, $\mathcal{M}=\phi^{-1}(\{0\})$. If every point in $\mathcal{M}$ is regular, i.e., the tangent map $T_{p}\phi$ is surjective for every $p\in \mathcal{M}$, then $\mathcal{M}$ is a submanifold of $TQ$ with dimension $2n-m$ by the regular level set theorem. 

Now let $\phi = (\phi^{1}, \dots, \phi^{m})$ denote the coordinate functions of the constraint $\phi$. Considering the dual of the canonical almost tangent structure $J$, we have that $J^{*}(d\phi^{a}) = \frac{\partial \phi^{a}}{\partial \dot{q}^{i}}dq^{i}$. Notice also that
$J^{*}(d\phi^{a})(X^{V}) = 0$. The equations of motion are integral curves of a vector field $\Gamma_{nh}$ defined by the equations
\begin{equation}\label{noneq}
    \begin{split}
        & i_{\Gamma_{nh}}\omega_{L} - dE_{L} = \lambda_{a}J^{*}(d\phi^{a}) \\
        & \Gamma_{nh} \in TM,
    \end{split}
\end{equation}
where $\lambda_{a}$ are Lagrange multiplier's to be determined.

These equations have a well-defined solution if $\sharp_{\omega_{L}}(J^{*}(d\phi^{a}))\cap TM = \{0\}$. Moreover, the expression in coordinates of integral curves of $\Gamma_{nh}$ are called Chetaev's equations:
\begin{equation}
    \begin{split}
        & \frac{d}{dt}\left(\frac{\partial L}{\partial \dot{q}}\right)-\frac{\partial L}{\partial q}=\lambda_{a} \frac{\partial \phi^{a}}{\partial \dot{q}} \\
        & \phi^{a}(q,\dot{q}) = 0
    \end{split}
\end{equation}
and they are the equations of motion for systems with nonlinear constraints (see \cite{paula}, \cite{cendra}, \cite{MdLeon} for more details).

In the following, we will consider a slight generalization of the concept of distribution. We will consider a mapping that to each point $v_{q}$ on the submanifold $\mathcal{M}$ assigns a vector subspace of $T_{v_{q}}(TQ)$. This map is a distribution on $TQ$ restricted to $\mathcal{M}$. From now on, let $S$ be a distribution on $TQ$ restricted to $\mathcal{M}$, whose annihilator is spanned by the one-forms $J^{*}(d\phi^{a})$, i.e.,$$S^{o}=\left\langle \{J^{*} (d\phi^{a})\}\right\rangle$$

Chetaev's equations may be written in Riemannian form using a geodesic-like equation according to the following theorem:

\begin{theorem}
    A curve $q:I\rightarrow Q$ is a solution of Chetaev's equations for a mechanical type Lagrangian if and only if $\phi(q,\dot{q})=0$ and it satisfies the equation
    \begin{equation}\label{Chetaev's eqns}
        \left( \nabla_{\dot{q}}\dot{q} + \text{grad} V \right)^{V} \in S^{\bot},
    \end{equation}
    where $S^{\bot}$, the orthogonal distribution to $S$, with respect to the semi-Riemannian metric $\mathcal{G}^{c}$.
\end{theorem}


\begin{proof}
    Suppose the Lagrangian $L$ is determined by a Riemannian metric $\mathcal{G}$ on $Q$ and a potential function $V$, so that its local expression is
    $$L(q,\dot{q})=\frac{1}{2}\mathcal{G}_{ij}\dot{q}^{i}\dot{q}^{j} - V(q).$$

    Chetaev's equations consist of Euler-Lagrange equations plus a reaction force term responsible for enforcing the constraints. In local coordinates we eventually get
    $$\ddot{q}^{i} - \mathcal{G}^{ij}\left[ \frac{1}{2}\frac{\partial \mathcal{G}_{lk}}{\partial q^{j}}\dot{q}^{l}\dot{q}^{k}-\frac{\partial \mathcal{G}_{lj}}{\partial q^{k}}\dot{q}^{l}\dot{q}^{k} - \frac{\partial V}{\partial q^{j}}\right] = \lambda_{a}\mathcal{G}^{ij}\frac{\partial \phi^{a}}{\partial \dot{q}^{j}},$$
    where $\mathcal{G}^{ij}$ is the inverse matrix of $\mathcal{G}_{ij}$. The left-hand side can be recognized to be the coordinate expression of the vector field
    $$\nabla_{\dot{q}}\dot{q} + \text{grad} V$$
    (see \cite{B&L} for details). We will show that the right-hand side is the coordinate expression of the vector field $\sharp_{\mathcal{G}^{c}}(J^{*}(d\phi^{a}))$. 
    
    Given a one-form $\alpha$ on $TQ$, the inverse musical isomorphism $\sharp_{\mathcal{G}^{c}}(\alpha)$ is characterized by
    $$\mathcal{G}^{c}(\sharp_{\mathcal{G}^{c}}(\alpha), X) = \langle \alpha, X \rangle, \quad \text{for any } X\in \mathfrak{X}(TQ).$$

    Using this property, and taking  into account the coordinate expression of $J^{*}(d\phi^{a})$, we can deduce from
    \begin{equation*}
        \begin{cases}
            \langle J^{*}(d\phi^{a}), \frac{\partial}{\partial q^{j}}\rangle & = \frac{\partial \phi^{a}}{\partial \dot{q}^{j}}\\
            \langle J^{*}(d\phi^{a}), \frac{\partial}{\partial \dot{q}^{j}}\rangle & = 0
        \end{cases}
    \end{equation*}
    that $\sharp_{\mathcal{G}^{c}} (J^{*}(d\phi^{a})) = \mathcal{G}^{ij}\frac{\partial \phi^{a}}{\partial \dot{q}^{j}}\frac{\partial}{\partial \dot{q}^{j}} \in \sharp_{\mathcal{G}^{c}}(S^{o})$.  In addition, we have that $S^{\bot}$ satisfies $S^{\bot} = \sharp_{\mathcal{G}^{c}}(S^{o})$. 

    Thus, using the coordinate expression of the vertical lift we deduce that
    $$(\nabla_{\dot{q}}\dot{q} + \text{grad} V)^{V} = \lambda_{a} \sharp_{\mathcal{G}^{c}}J^{*}(d\phi^{a}),$$
    which finishes the proof.
\end{proof}

\begin{remark}
 Notice that $S^{\bot}$ is spanned by vertical vectors in $TQ$. This observation will be relevant later in the paper.\hfill$\diamond$
\end{remark}

\section{Virtual nonholonomic constraints}\label{sec:controler}
Next, we present the rigorous construction of virtual nonholonomic  constraints. In contrast to the case of standard constraints on mechanical systems, the concept of virtual constraint is always associated with a controlled system and not just with a submanifold defined by the constraints. 

Given an external force $F^{0}:TQ\rightarrow T^{*}Q$ and a control force $F:TQ\times U \rightarrow T^{*}Q$ of the form
\begin{equation}
    F(q,\dot{q},u) = \sum_{a=1}^{m} u_{a}f^{a}(q)
\end{equation}
where $f^{a}\in \Omega^{1}(Q)$ with $m<n$, $U\subset\mathbb{R}^{m}$ the set of controls and $u_a\in\mathbb{R}$ with $1\leq a\leq m$ the control inputs, consider the associated mechanical control system of the form
\begin{equation}\label{mechanical:control:system}
    \nabla_{\dot{q}}\dot{q} =Y^0(q,\dot{q})+u_{a}Y^{a}(q).
\end{equation}
where $Y^0(q,\dot{q})=\sharp_{\mathcal{G}} (F^0(q, \dot{q}))$ and $Y^{a}=\sharp_{\mathcal{G}} (f^{a}(q)).$

Hence, the solutions of the previous equation are the trajectories of a vector field of the form
\begin{equation}\label{SODE}\Gamma(q, \dot{q}, u)=G(q,\dot{q})+u_{a}(Y^{a})_{(q,\dot{q})}^{V}.\end{equation}
We call each $Y^{a}=\sharp(f^{a})$ a control force vector field, $G$ is the vector field determined by the unactuated forced mechanical system
\begin{equation*}
    \nabla_{\dot{q}}\dot{q} =Y^0(q,\dot{q}).
\end{equation*}

\begin{definition}
    The distribution $\mathcal{F}\subseteq TQ$ generated by the vector fields  $Y^{a}=\sharp_{\mathcal{G}}(f^{a})$ is called the \textit{input distribution} associated with the mechanical control system \eqref{mechanical:control:system}.
\end{definition}

Now we will define the concept of virtual nonholonomic constraint.

\begin{definition}
A \textit{virtual nonholonomic constraint} associated with the mechanical control system \eqref{mechanical:control:system} is a controlled invariant submanifold $\mathcal{M}\subseteq TQ$ for that system, that is, 
there exists a control function $\hat{u}:\mathcal{M}\rightarrow \mathbb{R}^{m}$ such that the solution of the closed-loop system satisfies $\psi_{t}(\mathcal{M})\subseteq \mathcal{M}$, where $\psi_{t}:TQ\rightarrow TQ$ denotes its flow.
\end{definition}



\begin{definition}
    Two subspaces $W_1$ and $W_2$ of a vector space $V$ are transversal if 
    \begin{enumerate}
        \item $V = W_1+W_2$
        \item $\dim V =\dim W_1 + \dim W_2$, i.e. the dimensions of $W_1$ and $W_2$ are complementary with respect to the ambient space dimension.
    \end{enumerate}
\end{definition}
\begin{theorem}\label{main:theorem}
If the tangent space, $T_{v_{q}}\mathcal{M}$, of the manifold $\mathcal{M}$ and the vertical lift of the control input distribution $\mathcal{F}$ are transversal and $T_{v_{q}}\mathcal{M}\cap \mathcal{F}^V=\{0\}$, then there exists a unique control function making $\mathcal{M}$ a virtual nonholonomic constraint associated with the mechanical control system \eqref{mechanical:control:system}.
\end{theorem}

\begin{proof}
    Suppose that $TTQ=T\mathcal{M}\oplus \mathcal{F}^V$ and that  trajectories of the control system \eqref{mechanical:control:system} may be written as the integral curves of the vector field $\Gamma$ defined by \eqref{SODE}. For each $v_{q}\in \mathcal{M}_{q}$, we have that $$\Gamma(v_{q})\in T_{v_{q}}(TQ)=T_{v_{q}}\mathcal{M}\oplus \hbox{span}\Big{\{}(Y^{a})_{v_{q}}^{V} \Big{\}},$$ with $Y^{a}=\sharp(f^{a})$. Using the uniqueness decomposition property arising from transversality, we conclude there exists a unique vector $\tau^{*}(v_{q})=(\tau_{1}^{*}(v_{q}),\cdots, \tau_{m}^{*}(v_{q}))\in \mathbb{R}^{m}$ such that $$\Gamma(v_{q})=G(v_{q})+\tau_{a}^{*}(v_{q})(Y^{a})_{v_{q}}^{V}\in T_{v_{q}}\mathcal{M}.$$ If $\mathcal{M}$ is defined by $m$ constraints of the form $\phi^{b}(v_{q})=0$, $1\leq b\leq m$, then the condition above may be rewritten as $$d\phi^{b}(G(v_{q})+\tau_{a}^{*}(v_{q})(Y^{a})_{v_{q}}^{V})=0,$$ which is equivalent to $$\tau_{a}^{*}(v_{q})d\phi^{b}((Y^{a})_{v_{q}}^{V})=-d\phi^{b}(G(v_{q})).$$ Note that, the equation above is a linear equation of the form $A(v_{q})\tau=b(v_{q})$, where $b(v_{q})$ is the vector $(-d\phi^{1}(G(v_{q})), \dots, -d\phi^{m}(G(v_{q})))\in \mathbb{R}^{m}$ and $A(v_{q})$ is the $m\times m$ matrix with entries $A^{b}_{a}(v_{q})=d\phi^{b}((Y^{a})_{v_{q}}^{V})= \frac{\partial\phi^b}{\partial\dot{q}}(q,\dot{q})(Y^{a})$, where the last equality may be deduced by computing the expressions in local coordinates. That is, if $(q^{i}, \dot{q}^{i})$ are natural bundle coordinates for the tangent bundle, then
    \begin{equation*}
        \begin{split}
            d\phi^{b}((Y^{a})_{v_{q}}^{V}) & = \left(\frac{\partial \phi^{b}}{\partial q^{j}}dq^{j} + \frac{\partial\phi^b}{\partial \dot{q}^i}d\dot{q}^{i}\right)\left(Y^{a,k}\frac{\partial}{\partial \dot{q}^{k}}\right) \\
            & = \frac{\partial\phi^b}{\partial \dot{q}^i}Y^{a,i} = \frac{\partial\phi^b}{\partial \dot{q}}(q,\dot{q})(Y^{a}).
        \end{split}
    \end{equation*}
    In addition, $A(v_{q})$ has full rank, since its columns are linearly independent. In fact suppose that
    \begin{equation*}
        c_{1}\begin{bmatrix} \frac{\partial\phi^1}{\partial \dot{q}}(Y^{1}) \\
        \vdots \\
        \frac{\partial\phi^m}{\partial \dot{q}}(Y^{1}) \end{bmatrix} + \cdots + c_{m}\begin{bmatrix} \frac{\partial\phi^1}{\partial \dot{q}}(Y^{m}) \\
        \vdots \\
        \frac{\partial\phi^m}{\partial \dot{q}}(Y^{m}) \end{bmatrix}= 0,
    \end{equation*}
    which is equivalent to
    \begin{equation*}
        \begin{bmatrix} \frac{\partial\phi^1}{\partial \dot{q}}(c_{1}Y^{1}+\cdots + c_{m}Y^{m}) \\
        \vdots \\
        \frac{\partial\phi^m}{\partial \dot{q}}(c_{1}Y^{1}+\cdots + c_{m}Y^{m}) \end{bmatrix}=0.
    \end{equation*}
    Moreever, by transversality we have $T_{v_q}\mathcal{M}\cap \mathcal{F}^V = \{0\}$ which implies that $c_{1}Y^{1}+\cdots + c_{m}Y^{m}=0$. Since $\{Y_{i}\}$ are linearly independent we conclude that $c_{1}=\cdots=c_{m}=0$ and $A$ has full rank. But, since $A$ is an $m\times m$ matrix, and $\mathcal{M}$ is a constrained submanifold, it must be invertible. Therefore, there is a unique vector $\tau^{*}(v_{q})$ satisfying the matrix equation and $\tau^{*}:\mathcal{M}\rightarrow \mathbb{R}^{m}$ is smooth since it is the solution of a matrix equation depending smoothly on $v_{q}$.
\end{proof}


\begin{remark} In previous studies, virtual nonholonomic constraints were defined in somewhat different wyays. The most general one, comprising every single other as a particular case, is given in \cite{moran2021energy} where a virtual nonholonomic constraint is a set of the form $\mathcal{M}=\{(q,p)\in Q\times \mathbb{R}^{n} \ | \ h(q,p)=0\}$,
for which there exists a control law making it invariant under the flow of the closed-loop controlled Hamiltonian equations. This constraint may be rewritten using the cotangent bundle $T^{*}Q$ and $h$ may be seen as a function $h:T^{*}Q\rightarrow \mathbb{R}^{m}$. In addition, $h$ should satisfy $\text{rank } dh(q,p) = m$ for all $(q,p)\in \mathcal{M}$.

Our definition falls under this general definition. In order to see this, we must rewrite the virtual nonholonomic constraints and the control system on the cotangent bundle.

Indeed, consider the Hamiltonian function $H:T^{*}Q \rightarrow \mathbb{R}$ obtained from a Lagrangian function in the following way
$$H(q,p)=p\dot{q}(q,p)-L(q,\dot{q}(q,p)),$$
where $\dot{q}(q,p)$ is a function of $(q,p)$ given by the inverse of the Legendre transformation
$p=\frac{\partial L}{\partial \dot{q}}$.
The controlled Hamiltonian equations are given by
$$\dot{q}=\frac{\partial H}{\partial p}, \quad \dot{p}=-\frac{\partial H}{\partial q} + F^{0}(q,\dot{q}(q,p)) + u_{a}f^{a}(q),$$
where $F^{0}$ is an external force map. Now, any submanifold $\mathcal{M}\subseteq TQ$ might be defined as the set
$$\mathcal{M}= \{ (q,\dot{q})\in TQ \ | \ \phi^{a}(q,\dot{q}) = 0\},$$
where $d\phi^{a}$ with $1 \leqslant a\leqslant m$ are $m$ linearly independent constraints. The cotangent version of the constraint manifold is the set
$$\tilde{\mathcal{M}}= \{ (q,p) \ | \ \phi^{a}(q,\dot{q}(q,p)) = 0 \}. $$
Therefore, we set $$h(q,p)=(\phi^{1}(q,\dot{q}(q,p)),\cdots, \phi^{m}(q,\dot{q}(q,p))).\hfill\diamond $$ 
\end{remark}


\begin{example} \label{ex 1}
    {Consider a particle moving in three dimensional space and subject to the gravitational potential}. Its configuration space is $Q=\R^3$ with $q=(x,y,z)\in Q.$ The Lagrangian $L:TQ\rightarrow\R$, is given by 
    $$L(q,\dot{q})=\frac{m}{2}\left(\dotx^2+\doty^2+\dotz^2\right)-mgz,$$
    and we consider the constraint that is imposed by  $\Phi(q,\dot{q})=0$ with \[\Phi(q,\dot{q})=a^2\left(\dotx^2+\doty^2\right)-\dotz^2\] and the constraint manifold defined as \[\mathcal{M}=\{(q,\dot{q})\in TQ \;: \; \Phi(q,\dot{q})=0\}.\] Consider also the control force $F:TQ\times U\to T^*Q$
    \[F(q,\dot{q},u)=uf=u\left(xdx+ydy+dz\right).\]
    The controlled Euler-Lagrange equations are \[m\ddot{x}=ux, \quad m\ddot{y}=uy, \quad m\ddot{z}=-gm +u.\]
    The tangent space of the constraint manifold $\mathcal{M}$ is given by
    \begin{align*}
        T_{(q,\dot{q})}\mathcal{M}&=\{v\in TTQ\; :\; d\Phi(v)=0\}\\
         & =\spn\{X_1, X_2, X_3, X_4, X_5\},
    \end{align*}
    where $(q,\dot{q})\in\mathcal{M}$ and
\[X_1=\frac{\partial}{\partial x},  \quad X_2=\frac{\partial}{\partial y}, \quad X_3=\frac{\partial}{\partial z}, \quad\]
        \[X_4=\dotz\frac{\partial}{\partial \doty}+a^2\doty\frac{\partial}{\partial \dotz} \quad        X_5=\dotz\frac{\partial}{\partial \dotx}+a^2\dotx\frac{\partial}{\partial \dotz}.\]
    and the input distribution $\mathcal{F}$ is generated by the vector field
    \[Y=\frac{x}{m}\frac{\partial}{\partial x}+\frac{y}{m}\frac{\partial}{\partial y}+\frac{1}{m}\frac{\partial}{\partial z}.\]

    The control law that makes the constraint manifold invariant is given by 
    \[\hat{u}=-\frac{mg\dotz}{a^2x\dotx+a^2y\doty-\dotz}.\]   
\end{example}

In the following, we characterize the closed-loop dynamics as solutions of the nonholonomic equations \eqref{noneq}.

\begin{theorem}
    A curve $q:I\rightarrow Q$ is a trajectory of the closed-loop system for the Lagrangian control system \eqref{mechanical:control:system} making $\mathcal{M}$ invariant if and only if it satisfies
        \begin{equation}\label{constrained:equation}
        (\nabla_{\dot{q}}\dot{q} + \text{grad} V)^{V} =-\tau^*_a ( \sharp_{\mathcal{G}^{c}}(f^a)^{V}),
    \end{equation}
    or, in other words,
    \begin{equation}
        (\nabla_{\dot{q}}\dot{q} + \text{grad} V )^{V} \in \mathcal{F}^V
    \end{equation}
    
\noindent where $\mathcal{F}^V$ is the distribution on $TQ$ spanned by the vector fields $\{\sharp_{\mathcal{G}^{c}}(f^a)^{V}\}$ and $\tau^*_a$ the unique control from Theorem \ref{main:theorem}.
\end{theorem}

\begin{proof}
Let $q:I\to Q$ be the trajectory of the mechanical system (\ref{mechanical:control:system}), hence it is an intergral curve of the vector vield $\Gamma(v_q)$ with $v_q\in TQ$, of the form (\ref{SODE})
\begin{equation*}\Gamma(v_{q})=G(v_{q})+u_{a}(Y^{a})_{v_{q}}^{V}.\end{equation*}
From Theorem \ref{main:theorem} there exists a unique control function $\tau^*_a$ that makes $\mathcal{M}$ a virtual nonholonomic constraint i.e. \[\Gamma(v_{q}) = G(v_{q})+\tau_{a}^{*}(v_{q})(Y^{a})_{v_{q}}^{V}\in T_{v_{q}}\mathcal{M}\]

By the observations preceding equation \eqref{SODE} the trajectories of $\Gamma$ satisfy the equation 
$\nabla_{\dot{q}}\dot{q} + \text{grad} V + \tau^*_a Y^a  = 0$, 
since $Y^a=\sharp_\mathcal{G}(f^a)$ and lifting the equation yields

$$(\nabla_{\dot{q}}\dot{q} + \text{grad} V)^{V} = - \tau^{*}_a( \sharp_\mathcal{G}(f^a))^{V}.$$

By Lemma \ref{completemetricLemma} of the complete lift of the Riemannian metric $\mathcal{G}$, $\mathcal{G}^c$, we have $(\nabla_{\dot{q}}\dot{q} + \text{grad} V)^{V} = - \tau^{*}_a ( \sharp_{\mathcal{G}^{c}}(f^a)^{V})$. 
\end{proof}


The next proposition shows that if the vertical lift of the input distribution is orthogonal to the tangent space $T_{v_q}\mathcal{M}$ of the virtual nonholonomic constraint manifold $\mathcal{M}$ then the constrained dynamics is precisely the nonholonomic dynamics with respect to the original Lagrangian function.

\begin{proposition}\label{orthogonal:input:distribution}
If $\mathcal{F}^V$ is equal to $S^{\bot}$  then the trajectories of the feedback controlled mechanical system \eqref{constrained:equation} are the nonholonomic equations of motion \eqref{Chetaev's eqns}.
\end{proposition}


\begin{proof}
From the Chetaev's equations (\ref{Chetaev's eqns}) we have that the vector field $(\nabla_{\dot{q}}\dot{q} + \text{grad} V)^V$ is a linear combination of $\sharp_{\mathcal{G}^c}(J^*(d\phi^a))$ the generators of $S^\perp$. If $S^\perp$ equals $\mathcal{F}^V$ then $(\nabla_{\dot{q}}\dot{q} + \text{grad} V)^V\in\mathcal{F}^V$ which yields equation (\ref{constrained:equation}). 
\end{proof}

\begin{remark}
    Notice that given a mechanical system with nonlinear constraints, there always exist a distribution $\mathcal{F}$ such that $\mathcal{F}^{V}=S^{\bot}$, since $S^{\bot}$ is spanned by vertical lifts of vector fields on $Q$.\hfill$\diamond$
\end{remark}

The next example illustrates Proposition \ref{orthogonal:input:distribution}.

\begin{example}
      Consider, as in example \ref{ex 1}, {a particle moving in three dimensional space and subject to the gravitational potential}, with the same Lagrangian $L:TQ\to\R$, 
    \[L(q,\dot{q})=\frac{m}{2}\left(\dotx^2+\doty^2+\dotz^2\right)-mgz\]
    but consider now a constraint that makes the magnitude of the velocity constant, namely, $\Phi(q,\dot{q})=0$ with \[\Phi=\dotx^2+\doty^2+\dotz^2-c=0, \quad c>0.\] The constraint manifold is given by \[\mathcal{M}=\{(q,\dot{q})\in TQ \;: \; \Phi(q,\dot{q})=0\}\]  and  consider the control force $F:TQ\times U\to T^*Q$
    \[F(q,\dot{q},u)=uf=u(\dot{x}dx+\dot{y}dy+\dot{z}dz).\]
    The controlled Euler-Lagrange equations are \[m\ddot{x}=u\dot{x}, \quad m\ddot{y}=u\dot{y}, \quad m\ddot{z}=-gm +u\dot{z}.\]
     The input distribution, $\mathcal{F}$, is generated by the vector field
    \[Y=\frac{\dot{x}}{m}\frac{\partial}{\partial x}+\frac{\dot{y}}{m}\frac{\partial}{\partial y}+\frac{\dot{z}}{m}\frac{\partial}{\partial z},\] thus the vertical lift of the input distribution, $\mathcal{F}^V$, is generated by \[Y^V=\frac{\dot{x}}{m}\frac{\partial}{\partial \dot{x}}+\frac{\dot{y}}{m}\frac{\partial}{\partial \dot{y}}+\frac{\dot{z}}{m}\frac{\partial}{\partial\dot{z}}.\]
 
    The control that makes the constraint manifold invariant is 
    \[\hat{u}=\frac{mg\dotz}{c}.\]
    For $S^\perp$, the orthogonal to $S$, we write the differencial of $\Phi$, namely, $d\Phi=2\dot{x}d\dot{x}+2\dot{y}d\dot{y}+2\dot{z}d\dot{z}$ and its image through the dual of the canonical almost tangent structure $J=dq\otimes\frac{\partial}{\partial\dot{q}}$,
    \[J^*(d\Phi)=2\dot{x}dx+2\dot{y}dy+2\dot{z}dz.\]
    Hence, $S^\perp$ is generated by 
    \[\sharp_{\mathcal{G}^c}(J^*(d\Phi))=\frac{\dot{x}}{m}d\dot{x}+\frac{\dot{y}}{m}d\dot{y}+\frac{\dot{z}}{m}d\dot{z}.\]
    Notice that the vertical lift of the input distribution, $\mathcal{F}^V$, is equal to $S^\perp$ and from Proposition \ref{orthogonal:input:distribution}, the local expression of the equations (\ref{Chetaev's eqns}) and (\ref{constrained:equation}) will be the same. Indeed, the equations of the corresponding nonholonomic systems are
    \[\begin{cases}
        m\ddot{x}=\lambda\dot{x} \\
        m\ddot{y}=\lambda\dot{y} \\
        m\ddot{z}+mg=\lambda\dot{z} \\
        \dotx^2+\doty^2+\dotz^2-c=0
    \end{cases}, \] where $\lambda\in\R$ is a Lagrange multiplier to be determined using the constraints, while the equations of the controlled system with control force determined by $F$ are given by 
    $$\begin{cases}
        m\ddot{x}=-\tau^*\dot{x} \\
        m\ddot{x}=-\tau^*\dot{y} \\
        m\ddot{x}+mg=-\tau^*\dot{z}
    \end{cases},$$
    where $\tau^{*}$ is the unique feedback control making the constraints invariant under the flow. The two systems are equivalent on the submanifold $\mathcal{M}$ i.e. the trajectories of the constrained mechanical system \eqref{constrained:equation} and the nonholonomic equations of motion \eqref{Chetaev's eqns} coincide on the constraint manifold.
\end{example}

\begin{remark}
    Note that, in the example above, we are enforcing as constraints a constant value of the kinetic energy of a thermostat system - see \cite{Rojo Bloch} for a complementary analysis of this problem.\hfill$\diamond$
\end{remark}

\begin{remark}
    When $\mathcal{M}$ is a linear distribution on $Q$, the assumption made in the previous proposition reduces to the assumption considered in \cite{virtual}, i.e., $\mathcal{F}$ is orthogonal to $\mathcal{M}$.

    Indeed, suppose that $\mathcal{M}$ is a linear distribution $\mathcal{D}$. There exist one-forms $\{\mu^{a}\}$, $a=1, \dots, m$ such that $(\mu^{a})^{V}= J^{*}(d\phi^{a})$. Hence, $(\mathcal{D}^{o})^{V}=S^{o}$. Therefore, $S=\{X\in T(TQ) | (\mu^{a})^{V}(X)=0 \}$ is a rank $2n-m$ distribution on $TQ$ spanned by vector fields of the form $X^{c}$, $X^{V}$ and $Y^{V}$ where $X\in\Gamma(\mathcal{D})$ and $Y\in \Gamma(\mathcal{D}^{\bot})$, and $S^{\bot}$ is a rank $m$ distribution on $TQ$ spanned by $Y^{V}$, so that $S^{\bot}$ is actually contained in $S$. Therefore $S^{\bot}=(\mathcal{D}^{\bot})^{V}$. Thus, the assumption $\mathcal{F}^{V}= S^{\bot}$ reduces to $\mathcal{F}=\mathcal{D}^{\bot}$ or, equivalently, the input distribution $\mathcal{F}$ must be orthogonal to the linear distribution $\mathcal{D}$. This is precisely the assumption made in \cite{virtual}.\hfill$\diamond$
\end{remark}
    
\section{Application and Simulation Results}\label{application section}

The previous results on virtual nonlinear nonholonomic constraints can be used to enforce a desired relation between state variables through a linear control force whenever the interplay between forces and constraints satisfies our assumptions. In the following, we give a particular application of how a desired constraint can be enforced in the problem of the motion of particles moving with an alignment on the velocities as in  \cite{Bloch Rojo}. This application can be useful in imposing virtual constraints for flocking motion in multi-agent systems \cite{flocking,flocking2}.

Consider two particles moving under the influence of gravity and which we desire to constrain to move with parallel velocity. Suppose that the motion of the two particles evolves in a plane parametrized by $(x,z)$. The position of the particles is given by $q_1=(x_1,0,z_1)$ and $q_2=(x_2,0,z_2)$, respectively, so the configuration space can be considered as $Q=\R^4$ with $q=(q_1,q_2)\in Q$.

The Lagrangian $L:TQ\to\R,$ is given by
\[L(q,\dot{q})=\frac{1}{2}m_1\dot{q}_1^2 + \frac{1}{2}m_2\dot{q}_2^2 - G(q)\]
where $G(q)=m_1gz_1 + m_2gz_2$ is the potential energy due to gravity and $m_i, i=1,2$ are the masses of the particles, respectively. The constraint is given by the equation $\Phi:TQ\to\R,$
\[\Phi(q,\dot{q})=\dot{x_1}\dot{z_2} - \dot{x_2}\dot{z_1}\]
and the control force is just $F:TQ\times\R\to T^*Q$ given by 
\[F(q,\dot{q},u)=u(f_1dx_1 + f_2dz_1 + f_3dx_2 + f_4dz_2).\]
The controlled Euler-Lagrange equations are 
\begin{equation}
    \begin{split}
         m_1 \ddot{x}_1 &= uf_1, \quad
         m_1\ddot{z}_1 +m_1g = uf_2, \\
         m_2\ddot{x}_2 &= uf_3, \quad
         m_2 \ddot{z}_2 + m_2g = uf_4.
    \end{split}
\end{equation}
The constraint manifold is $\mathcal{M}=\{(q,\dot{q})\in TQ \; :\; \Phi(q,\dot{q})=0\}$ and its tangent space, at every point $(q,\dot{q})\in\mathcal{M}$, is given by        $T_{(q,\dot{q})}\mathcal{M}=\{v\in TTQ\; :\; d\Phi(v)=0\}
          =\spn\{X_1, X_2, X_3, X_4, X_5, X_6, X_7\}$, with
\[X_1=\frac{\partial}{\partial x_1},  \quad X_2=\frac{\partial}{\partial z_1}, \quad X_3=\frac{\partial}{\partial x_2}, \quad X_4=\frac{\partial}{\partial z_2}\]
        \[X_5=\dot{x}_2\frac{\partial}{\partial \dot{x}_1} + \dot{z}_2\frac{\partial}{\partial\dot{z_1}} + \dot{x}_1\frac{\partial}{\partial\dot{x_2}} + \dot{z}_1\frac{\partial}{\partial\dot{z_2}}, \quad\]
        \[X_6=\dot{z}_1\frac{\partial}{\partial \dot{x}_1} + \dot{x}_1\frac{\partial}{\partial\dot{z_1}} + \dot{z}_2\frac{\partial}{\partial\dot{x_2}} + \dot{x}_2\frac{\partial}{\partial\dot{z_2}}, \quad\]
        \[X_7=\dot{x}_1\frac{\partial}{\partial \dot{x}_1} + \dot{z}_1\frac{\partial}{\partial\dot{z_1}} - \dot{x}_2\frac{\partial}{\partial\dot{x_2}} - \dot{z}_2\frac{\partial}{\partial\dot{z_2}}. \quad\]
The input distribution $\mathcal{F}$ is generated by the vector field $\displaystyle{Y=\frac{f_1}{m_1}\frac{\partial}{\partial x_1} + \frac{f_2}{m_1}\frac{\partial}{\partial z_1} + \frac{f_3}{m_2}\frac{\partial}{\partial z_2} + \frac{f_4}{m_2}\frac{\partial}{\partial z_2}}.$
Note here that the vertical lift of the input distribution, $\mathcal{F}^V$, which is generated by $\displaystyle{Y^V=\frac{f_1}{m_1}\frac{\partial}{\partial \dot{x}_1} + \frac{f_2}{m_1}\frac{\partial}{\partial \dot{z}_1} + \frac{f_3}{m_2}\frac{\partial}{\partial \dot{z}_2} + \frac{f_4}{m_2}\frac{\partial}{\partial \dot{z}_2}},$ is transversal to the tangent space of the constraint manifold, $T\mathcal{M}$. By Theorem \ref{main:theorem} there is a unique control law making the constraint manifold a virtual nonholonomic constraint. 
    The control law that makes the constraint manifold invariant is
    \begin{equation*}
            \hat{u}=\left(\dot{z}_2f_1-\dot{z}_1f_3 + \dot{x}_1f_4-\dot{x}_2f_2\right)^{-1}\left( g\dot{x}_1 - g\dot{x}_2\right).
    \end{equation*}
    For $f_1=f_2=1$ and $f_3=f_4=0$ we get $F(q,\dot{q},u)=u(dx_1 + dz_1)$ and
    \begin{equation*}
            \hat{u}=\left(\dot{z}_2 -\dot{x}_2 \right)^{-1}\left( g\dot{x}_1 - g\dot{x}_2\right).
    \end{equation*}

   We have simulated the closed-loop control system with the preferred feedback control law using a standard fourth-order Runge-Kutta method and initial points $(x_{1},x_{2}, z_{1}, z_{2}) = (1, 40, 0, 0)$ and initial velocities $(\dot{x}_{1}, \dot{x}_{2}, \dot{z}_{1}, \dot{z}_{2})=(80, 20, 40, 10)$. In Fig. \ref{traj} we show the controlled trajectories for both particles where can be seen the velocities' compliance with the constraint. The total energy of the system is depicted in Fig. \ref{energy} while the preservation of the constraint during the simulation time is shown in Fig. \ref{constraint}. Fluctuations of the values of the constraint function appear due to simulation computational process and are restricted to a minor interval as expected. The control function is depicted in Fig. \ref{controls} where it tends to zero since the motion tends to become vertical and gravity takes over.

    \begin{figure}[htb!]
        \centering
        \includegraphics[scale=0.45]{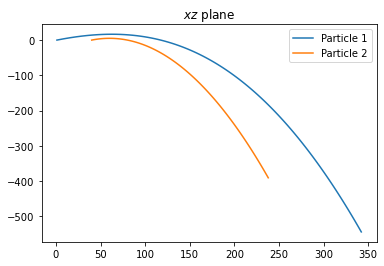}
        \caption{Controlled trajectory of the two particles}
        \label{traj}
    \end{figure}

    \begin{figure}[htb!]
        \centering
        \includegraphics[scale=0.45]{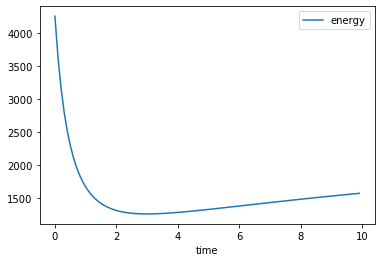}
        \caption{Total energy of the system}
        \label{energy}
    \end{figure}

    \begin{figure}[htb!]
        \centering
        \includegraphics[scale=0.45]{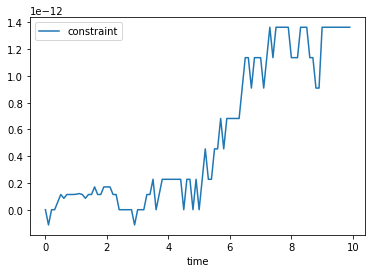}
        \caption{The value of the constraint function $\Phi$ during the simulation time}
        \label{constraint}
    \end{figure}

    \begin{figure}[htb!]
        \centering
        \includegraphics[scale=0.45]{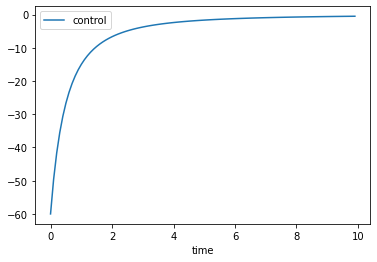}
        \caption{Control function during simulation time.}
        \label{controls}
    \end{figure}


\section{Conclusions and Future Work}

In this paper, we have extended our results in \cite{virtual} and \cite{affine} to the case of nonlinear constraints on the velocities. Our results guarantee that linear and affine control forces might be used to enforce desired constraints on the velocities and positions, provided they meet the assumptions on the statement of Theorem \ref{main:theorem}, that is, the tangent bundle to the constraint submanifold is transversal to the vertical lift of the input distribution. As future work, there is an obvious necessity of extending the range of applicability of our results to the cases in which the above assumptions are not met. In some of this cases, a control law might exist though it is possible that it is no longer unique.

One of our objectives for a future work is to adapt these results to applications to bipedal robot locomotion. To this end, we will extend our results to the same setting of the type of virtual nonlinear constraints appearing in \cite{griffin2015nonholonomic}. These constraints are of the form
$\phi(q_{a}, q_{u}, \dot{q}_{a}, \dot{q}_{u}) = q_{a} - h(q, \dot{q}_{u})$, where $q=(q_{a}, q_{u})$ are local coordinates of $Q$, with respect to which the actuated coordinate vector fields $\frac{\partial}{\partial q_{a}}$ are the control force vector fields $Y^{a}$, i.e, the controlled equations are of the type
$\nabla_{\dot{q}_{u}}\dot{q}_{u} =Y^0(q,\dot{q}), \text{ and } \nabla_{\dot{q}_a}\dot{q}_a=Y^0(q,\dot{q})+u_{a}\frac{\partial}{\partial q_{a}}$. Under this assumption, the vertical lift of $Y^{a}$ belongs to $T\mathcal{M}$, since $\langle d\phi, (Y^{a})^V \rangle = \langle d\phi, \frac{\partial}{\partial \dot{q}_{a}} \rangle=\frac{\partial \phi}{\partial \dot{q}_{a}}=0$. Therefore, $\mathcal{F}^{V}\subseteq T\mathcal{M}$. Hence, this set of virtual nonholonomic constraints does not fall  under the assumptions of Theorem \ref{main:theorem}. In addition, we will also address a related problem: instead of enforcing a constraint, we will study from the geometric point of view the stabilization properties of the virtual nonholonomic constraint.



\end{document}